%% file: main.tex
\documentclass[letterpaper,11pt,reqno]{amsart} 
\usepackage[portrait,margin=1in]{geometry} 
\usepackage{cite} 
\include{style}

\begin{document}
\title{Interface between competing random walks on a cycle}
\author[Chatterjee]{Shirshendu Chatterjee}
\address{
Department of Mathematics\\
The City College of New York and CUNY Graduate Center\\
New York, NY,
USA
}
\email{shirshendu@ccny.cuny.edu}
\author[Nabahi]{Nadya Nabahi}
\address{Nadya Nabahi\\
The City College of New York\\
New York, NY,
USA
}
\email{nnabahi000@citymail.cuny.edu}
\author[Terlov]{Grigory Terlov}
\address{
Grigory Terlov \\
Department of Statistics and Operations Research\\
University of North Carolina\\
Chapel Hill, NC \\
USA
}
\email{gterlov@unc.edu}

\begin{abstract}
We consider a competition between two independent random walks on a cycle of length $N$. Each vertex is claimed by the walker that visits it first, and remains claimed thereafter. We prove that if the initial distance between the walkers is $d$, then the expected number of edges whose endpoints are claimed by different walkers is of order $\ln(1+N/d).$
This confirms the logarithmic dependence on $N/d$ predicted in Gomes Jr.\ et al.\ \cite{GLDSH96}.
\end{abstract}
\maketitle
\section{Introduction}
Random walk is one of the simplest mathematical models of diffusion and a cornerstone of modern probability theory. The set of distinct vertices visited by the walker up to a certain time, known as the \textbf{range}, induces a random growth process. For a random walk on $\bZ^d$, the asymptotics of the mean and fluctuations of its range are well studied, e.g.\ \cite{DE51,JO68,JP71,JP70recur,JP70clt,LeGall,BCR09}, with the discrete torus $\bZ^d/n\bZ^d$ treated in \cite{DK21}, and percolative properties in \cite{Sznitman10, TW11}. Interpreting the range of a random walk as explored territory made it a quantity of interest in population ecology \cite{Bansaye24,benichou2014depletion,regnier2023universal,regnier2024maximum}. 

Motivated by the exploration interpretation of the random walk, a natural next step is to consider several walkers simultaneously discovering the network. The simplest models, where the walkers do not interact with each other, turn out to be already significantly more intricate than the classical setting and became an active subject of study only in the last few decades. To our knowledge, the first such model was the collaborative dynamic, introduced in \cite{Larralde92}, where a vertex is considered explored if at least one walker visited it. 
This direction subsequently received considerable attention through the study of cover and hitting times \cite{Alon_at_al,sauerwald2010expansion,patel2016hitting,ivaskovic2017multiple,elsasser2011tight,efremenko2009well}, motivated in part by distributed and parallel computation; recent probabilistic treatments include \cite{rivera2023multiple,Hermon25,CollabRW}.

In this paper, we are interested in the competition between random walks, where each vertex is labeled by the walker that visited it first. This model is a natural proxy for competition for territory and resources in ecological systems.
As with the collaborative model, it originated in the physics literature \cite{GLDSH96} and has since attracted considerable interest in probability theory \cite{Dicker06,Miller13,Baccara26}.
While most of the literature is concerned with the proportion of the graph claimed by each walker, few results address the geometry of the competing ranges or the relation between them. 
A natural starting point is to understand the border between 
the sets of vertices claimed by each walker, known as the \textbf{interface}.
Formally, for graph $G=(V,E)$ and a subset of edges on $A\subseteq E(G)$, the interface induced by the competition of two random walks in $A$, denoted by $\cI_{A}$, is defined as the subset of edges in $A$ whose endpoints are claimed by different walkers. 
Its size, denoted by $\abs{\cI_{A}}$ is then a counting statistic indicative of the geometry of the underlying random growth process. It was already considered in \cite{GLDSH96} for two random walks on a cycle, however due to technical difficulty of the problem, even in this arguably simple case, Gomes Jr.\ et al.\ opted for a simulation-based prediction.
They found, surprisingly, that for two walkers started at distance $d$ on a cycle of length $N$ the expected size of the interface is of order $\ln\left(1+{N}/{d}\right)$. The ranges therefore tend to interleave rather than dividing the cycle into two arcs, and the growth of the interface size depends on $N$ and $d$ only through their ratio. Our main result confirms their prediction.

\begin{thm}\label{thm:main}
     Let $d,N\in\bN$ such that $d\le N/2$ and $C_N$ be a cycle of length $N$. Consider two independent continuous-time simple symmetric random walks on $C_N$ started at distance $d$. Then there are universal constants $0<c\le C<\infty$ such that
    \[
      c\ln\left(1+\frac{N}{d}\right)\le \E\abs{\cI_{C_N}}\le C\ln\left(1+\frac{N}{d}\right).
    \]
\end{thm}
In \cite{GLDSH96}, Gomes et al.\ simplified the analysis by assuming that $d=N/2$ and computed, for each vertex on a cycle, the probabilities that it is visited first by each of the two walkers. Their approach relied on a Green's function representation of first-visit probabilities. They noted that a similar analysis might be used to establish their prediction for the interface size, but did not pursue this direction rigorously because of the resulting technical complications.
Our approach is based on a different idea. Namely, we introduce a symmetric random walk on $\bR$ whose number of zero-crossings is precisely the size of the interface. In addition, to avoid tie-breaking, we work in continuous rather than discrete time, but this is a matter of technical convenience only.

\subsection{Organization} In Section~\ref{sec:prelim} we recall the basics of the continuous-time simple symmentic random walk on $\bZ$, related notation, and derive several axillary bounds. In Section~\ref{sec:ATD-walk} we introduce our main tool, the arrival-time difference walk on $\bZ$, using which we then derive upper and lower bounds on the expected size of the interface on an interval $[0,M]\subset\bZ$. Finally, in Section~\ref{sec:main-proof} we use these bounds to derive the main result.

\subsection{Acknowledgments}

The last author thanks Tyler M.\ Gall for insightful conversations about this problem and providing simulations during the undergraduate research project done at the University of Illinois. 

S.C.\ was supported in part by the NSF grant DMS-2154564.

N.N.\ was supported by the Rich internship at the mathematics department of CCNY.

G.T.\ was supported in part by the RTG award grant (DMS-2134107) from the NSF.

\subsection{Disclosure of generative-AI tool use.}
During the preparation of this work, we used generative-AI systems (OpenAI ChatGPT 5.5 and 5.6) to assist with literature discovery, the exploration of proof strategies, and language editing. In particular, an AI-assisted review identified a flaw in an earlier version of Section~\ref{sec:ATD_upper_bdd} and suggested an anti-concentration approach that we implemented here. All AI-generated suggestions were carefully fleshed out and checked by the authors, who take full responsibility for the results and the correctness of the proofs.

\section{Preliminaries and basic estimates}\label{sec:prelim}
In this section we introduce notation for simple random walks on $\bZ$ and their hitting times, recall basic facts, and provide some of the technical estimates that will be needed in the following sections.

Let $(X_t)_{t\ge0}$ denotes the continuous-time simple symmetric random walk on $\bZ$ with jump times given by independent mean-one exponential random variables.  Thus, when the walk jumps, it moves one unit to the right or one unit to the left, each with probability $1/2$.  
Let $J_n$, for $n\ge1,$ be the successive jump times, yielding the discrete-time simple random walk $X_{J_n}$.
We write $\pr_0$ and $\E_0$ for probability and expectation when $X_0=0$.
For $v\in\bZ$, define the random time time $\tau_v$ and the number of jumps $\wh{\tau}_v$ that it takes for $X_t$ to reach $v$, that is
\[
  \tau_v:=\inf\{t\ge0:X_t=v\}
  \qquad\mbox{and}\qquad
  \wh{\tau}_v:=\inf\{n\ge0:X_{J_n}=v\}
\]
with the convention $\inf\emptyset=\infty$. It is a standard fact that conditional on $\{\wh{\tau}_1=n\}$, the random variable $\tau_1$ has $\Gamma(n,1)$ distribution. In particular, $\tau_1$ enjoys many useful properties such as has bounded density the following Laplace transform: for every $z\in\bC$ with $\mathrm{Re}(z)\ge0$,
\begin{equation}\label{eq:tau_transform}
    \E_0 e^{-z\tau_1}=1+z-\sqrt{z(2+z)}.
\end{equation}
Here the square root is the branch obtained by continuation from the positive real axis; equivalently, it has positive real part when $\mathrm{Re}(z)>0$ and is extended continuously to $\mathrm{Re}(z)=0$.

We next state several axillary lemmas bounding tail probabilities of $\tau_1$ and related random variables. Some of these inequalities could be stated more generally, rather than with particular choice of constants. However, since derivation of such inequalities is routine we opted to state the exact statements we need instead of the most general form.

\textit{Throughout the paper, $c$ and $C$ denote positive constants whose values may change from one expression to the next; we distinguish them by subscripts or prime only where additional clarity is needed.}

\begin{lem}\label{lem:basic_bounds}
Let $\eta$ and $\eta'$ be independent random variables, each with the same distribution as $\tau_1$ under $\pr_0$. Then there are constants $0<c<C<\infty$ such that, for every $t\ge1$,
\begin{enumerate}
\item 
\[
    c t^{-1/2}\le \pr(\eta>t)\le C t^{-1/2}.
\]
\item 
\[
  c t^{-1/2}\le\pr\left(\frac32 t<\eta\le 2 t\right)\le C t^{-1/2}.
\]
\item 
\[
  ct^{-1/2} \le \pr(\eta-\eta'>t)\le C t^{-1/2}.
\]
\item
\[
  \E\left((\eta-\eta')^2\ind_{\{|\eta-\eta'|\le t\}}\right)\le C t^{3/2}.
\]
\end{enumerate}
\end{lem}
\begin{proof}

It is classical that $\tau_1$ has density $f_{\tau_1}(t)=t^{-1}e^{-t}I_1(t),$ where $I_1(t)$ is the (modified) Bessel function, e.g.\ see \cite[eq.~(7.13), p.~60]{feller1991introduction}.
The standard asymptotic for $I_1(t)$ can be written as
\[
    \lim_{t\to\infty}\sqrt{2\pi t}\,e^{-t}I_1(t)=1,
\]
and thus
\[
    \lim_{t\to\infty}t^{3/2}f_{\tau_1}(t)=\frac{1}{\sqrt{2\pi}}.
\]
Since $f_{\tau_1}$ is continuous and strictly positive on $(0,\infty)$, this implies part (1). Part (2) follows immediately.
For part (3), observe that for $t\ge1$
\[
    ct^{-1/2}\le\pr(\eta>2t)\pr(\eta'\le1)\le\pr(\eta-\eta'>t)\le \pr(\eta>t)\le Ct^{-1/2}.
\]
Finally, by symmetry of $\eta-\eta'$ and part (3) we have that for $t\ge1$
\[
\pr(|\eta-\eta'|>t)=2\pr(\eta-\eta'>t)\le Ct^{-1/2}.
\]
Thus
\[
    \E\left((\eta-\eta')^2\ind_{\{\abs{\eta-\eta'}\le t\}}\right)
    \le2\int_0^t s\,\pr(|\eta-\eta'|>s)\,ds\le 2\int_0^1s\,ds+C\int_1^t s^{1/2}\,ds\le Ct^{3/2}.
\]
\end{proof}

Next lemma establishes a tail bound for the hitting times for $X_t$.

\begin{lem}\label{lem:tau-tail}
    There exist a universal constant $C>0$ such that for every $n\ge1$ and $x\ge1$ we have
    \[
    \pr(\tau_n>x)\le\frac{Cn}{\sqrt{x}}.
    \]
\end{lem}
\begin{proof}
    By translation invariance and the strong Markov property,
    $
    \tau_n\stackrel{\mathrm d}{=}\eta_1+\cdots+\eta_n,
    $
    where $\eta_1,\ldots,\eta_n$ are independent copies of $\tau_1$ under $\pr_0$. From Lemma~\ref{lem:basic_bounds}.(1), we get that for some universal constant $C>0$ and every $x\ge1$ 
    \begin{align*}
    \pr(\tau_n>x)&\le\pr\left(\max_{1\le k\le n}\eta_k>x\right)+\pr\left(\sum_{k=1}^n\eta_k\ind_{\{\eta_k\le x\}}>x\right)\\
    &\le n\pr(\eta>x)+\frac{n}{x}\E\left(\eta_1\ind_{\{\eta_1\le x\}}\right)\\
    &\le n\pr(\eta_1>x)+\frac{n}{x}\int_0^x\pr(\eta_1>t)\,dt
    \le\frac{Cn}{\sqrt{x}}.
    \end{align*}
\end{proof}

Denote the characteristic function of $\tau_1$ by $\varphi(\theta):=\E_0e^{i\theta\tau_1}$ for $\theta\in\bR$.

\begin{lem}\label{lem:char_func_decay}
There is a universal constant $C>0$ such that, for every integer $k\ge2$,
\[
\int_{\bR}\abs{\varphi(\theta)}^k\,d\theta
\le \frac{C}{k^2}.
\]
\end{lem}

\begin{proof} As above we adapt the convention where  the square root is taken with nonnegative real part.
From \eqref{eq:tau_transform}, with $z=i\theta$ for $\theta\in\bR$, it follows that there are positive constants $c$ and $\delta$ such that for any $\abs{\theta}\le\delta$ the inequality $\abs{\varphi(\theta)}\le\exp(-c\sqrt{\abs{\theta}})$ holds.
Hence
\[
\int_{-\delta}^{\delta}\abs{\varphi(\theta)}^k\,d\theta\le C\int_0^\infty e^{-ck\sqrt{\theta}}\,d\theta\le\frac{C_1}{k^2}.
\]

Since $\tau_1$ has a absolutely continuous distribution, $\abs{\varphi(\theta)}<1$ for $\theta\neq0$.
In particular, on every compact set bounded away from $0$, $\abs{\varphi}$ is bounded above by a constant strictly smaller than $1$.

Finally, rationalizing the expression in \eqref{eq:tau_transform} gives $\abs{\varphi(\theta)}\le {C_2}/{\abs{\theta}}$
for all sufficiently large $\abs{\theta}$. 
Choose a constant $R$ large enough such that this inequality holds and $R>C_2$, choose $\delta$ as above, and define
$$
\rho:=\sup_{\delta<\abs{\theta}\le R}\abs{\varphi(\theta)}\in(0,1).
$$
Then
\begin{align*}
    \int_{\bR}\abs{\varphi(\theta)}^k\,d\theta&=\int_{-\delta}^\delta\abs{\varphi(\theta)}^k\,d\theta+\int_{\delta<\abs{\theta}<R}\abs{\varphi(\theta)}^k\,d\theta+\int_{\abs{\theta}>R}\abs{\varphi(\theta)}^k\,d\theta\\
    &\le \frac{C_1}{k^2}+2R\rho^k+2\int_R^\infty \left(\frac{C_2}{\theta}\right)^k\,d\theta= \frac{C_1}{k^2}+2R\rho^k+\frac{2R}{k-1}\cdot\left(\frac{C_2}{R}\right)^k\le\frac{C}{k^2}
\end{align*}
for some universal constant $C>0$.
\end{proof}

\section{The arrival-time difference walk}\label{sec:ATD-walk}
In this section we consider the competition between two random walks on $\bZ$.
We denote two independent random walks by $X^{(1)}_t$ and $X^{(2)}_t$, which start at 0 and $-d$ respectively, and are interested in the expected size of the interface $\cI_{[0,M]}$ on the interval $[0,M]$, for some $M\in\bN$.
Recalling notation from Section~\ref{sec:prelim}, let $\tau^{(i)}_n$ denote the first time the respective random walk visits vertex $n$, and note that by the strong Markov property and independence between the walks $X^{(i)}_t$, the random variables $\eta^{(i)}_n:=\left(\tau^{(i)}_{n}-\tau^{(i)}_{n-1}\right)$ for $i\in \{1,2\}$ and $n\in\bN$ are independent, identically distributed, and have the same distribution as $\tau_1$ under $\pr_0$.

Note that a vertex $n$ is claimed by the $X^{(1)}_t$ if and only if $\tau^{(1)}_{n}-\tau^{(2)}_{n}<0$. This difference be rewritten as 
\[
    \tau^{(1)}_{n}-\tau^{(2)}_{n}=\sum_{k=1}^{n}\left(\tau^{(1)}_{k}-\tau^{(1)}_{k-1}\right)-\sum_{k=1}^{n}\left(\tau^{(2)}_{k}-\tau^{(2)}_{k-1}\right)-\tau^{(2)}_{0}=\sum_{k=1}^{n}\left(\eta^{(1)}_k-\eta^{(2)}_{k}\right)-\tau^{(2)}_{0}.
\]
We interpret it as a shifted, by a random quantity $\tau^{(2)}_{0}$, symmetric discrete-time random walk on $\bR$ with i.i.d.\ steps that have the same distribution as $\tau_1$ and its independent copy $\tau_1'$ under $\pr_0$. To formalize the notation let $\zeta_n:=\eta_n^{(1)}-\eta_n^{(2)}$, $n\ge1$, and these are i.i.d.\ random variables with the same distribution as $(\tau_1-\tau_1')$. Then \textbf{the arrival-time difference walk (ATD walk)} $S_n$ is defined by
\begin{equation}\label{eq:ATD walk}
     S_n=\sum_{k=1}^n\zeta_k.
\end{equation}

For a random variable $U$, the we denote the shifted ATD walk by 
\begin{equation}\label{eq:shifted ATD walk}
    Y_{U,n}:= S_n-U.
\end{equation}
In particular, when $U=\tau^{(2)}_{0}$ we have $Y_{U,n}=\tau^{(1)}_{n}-\tau^{(2)}_{n}$.

Since hitting times have continuous distributions, and thus almost surely distinct, we can rewrite the size of the interface as the number of zero-crossings of the shifted ATD walk
\begin{equation}\label{eq:ATD-crossing-repr}
    \abs{\cI_{[0,M]}}=\sum_{n=1}^{M} \ind_{Y_{U,n-1}Y_{U,n}<0}, \quad\mbox{where}\quad U=\tau^{(2)}_{0}.
\end{equation}

The goal of this section is to prove the following theorem.
\begin{thm}\label{thm:ATDwalk}
    Let $d,M\in\bN$ such that $d<M$. Suppose $X^{(1)}_t$ and $X^{(2)}_t$ are two independent continuous time random walks on $\bZ$ started at $0$ and $-d$, respectively. Then there are universal constants $c,C>0$ such that
    \[
      c\ln\left(1+\frac{M}{d}\right)\le \E\abs{\cI_{[0,M]}}\le C\ln\left(1+\frac{M}{d}\right).
    \]
\end{thm}
We prove Theorem~\ref{thm:ATDwalk} by showing the upper and lower bounds separately in Propositions~\ref{prop:ATD_upper_bdd} and~\ref{prop:ATD_lower_bdd}, respectively, to which the next two subsections are dedicated. For the purposes of proving Theorem~\ref{thm:main}, one might be tempted to think that we would be able to use the bounds from Theorem~\ref{thm:ATDwalk} directly on the larger arc between the initial positions of the two walkers on the cycle. That is indeed the case with the upper bound. For the lower bound, however, we were not able to use this bound as is and instead will use a similar argument and some of preliminary lemmas developed in Subsection~\ref{sec:ATD_lower_bdd}.

\subsection{Upper bound in Theorem~\ref{thm:ATDwalk}}\label{sec:ATD_upper_bdd}
To derive the upper bound in Theorem~\ref{thm:ATDwalk} we prove a slightly stronger statement than needed, removing the condition $d<M$. 
\begin{prop}\label{prop:ATD_upper_bdd}
    In the situation of Theorem~\ref{thm:ATDwalk}, without assuming $d<M$, there is a universal constant $C>0$ such that 
    \[
    \E\abs{\cI_{[0,M]}}\le C\ln\left(1+\frac{M}{d}\right).
    \]
\end{prop}

First, we establish the key anti-concentration estimate.

\begin{lem}\label{lem:ATD_density}
For $n\ge0$ and $d\ge1$, let $Y_{U,n}$ be the shifted ATD walk as in \eqref{eq:shifted ATD walk} with $U=\tau^{(2)}_{0}$. Then the probability density function $f_{n}(x)$ of $Y_{U,n}$ for all $x\in \bR$ satisfies
\[
f_{n}(x)\le\frac{C}{(n+d)^2},
\]
for some universal constant $C>0$.
\end{lem}
\begin{proof} We note that $f_{n}(x)$ depends on $d$, however for clarity we omitted it from the notation.
Recalling the notation $\varphi(\theta)$, since $U=\tau_0^{(2)}$ and is independent of $( S_n)_{n\ge0}$, we have that the characteristic function of $Y_{U,n}$ is given by
\[
\E e^{i\theta Y_{U,n}}=\varphi(\theta)^n\varphi(-\theta)^{n+d}.
\]
Now since $\varphi(-\theta)=\overline{\varphi(\theta)}$,
\[
\abs{\E e^{i\theta Y_{U,n}}}=\abs{\varphi(\theta)}^{2n+d}.
\]

Suppose that $2n+d\ge2$. Then by Fourier inversion theorem and Lemma~\ref{lem:char_func_decay}, for every $x\in\bR$
\begin{align*}
f_{n}(x)&\le\frac{1}{2\pi}\int_{\bR}\abs{\varphi(\theta)}^{2n+d}\,d\theta\le\frac{C}{(2n+d)^2}\le\frac{C}{(n+d)^2}.
\end{align*}
The only remaining case is $n=0$ and $d=1$.  In this case $X_0^{(2)}=-1$ and $Y_{U,0}=-U=-\tau_0^{(2)}$ has the same density, up to a reflection, as $\tau_1$, which is bounded.
\end{proof}

\begin{proof}[Proof of Proposition~\ref{prop:ATD_upper_bdd}]
    Since the distribution of $\zeta_n$ is symmetric, for every $x\ne0$,
    \[
        \pr\left(Y_{U,n-1}Y_{U,n}<0\mid Y_{U,n-1}=x\right)=\pr(\zeta_n>\abs{x}).
    \]
By Lemma~\ref{lem:basic_bounds}.(3)
\[
\pr(Y_{U,n-1}Y_{U,n}<0)\le C\E\left(\frac{1}{\sqrt{1+\abs{Y_{U,n-1}}}}\right).
\]

Then from Lemma~\ref{lem:ATD_density}, since $a:=n-1+d\ge1$, it follows that
\begin{align*}
\E\left(\frac{1}{\sqrt{1+\abs{Y_{U,n-1}}}}\right)&=\int_{\bR}\frac{f_{n-1}(x)}{\sqrt{1+\abs{x}}}\,dx\le\frac{C}{a^2}\int_{-a^2}^{a^2}\frac{dx}{\sqrt{1+\abs{x}}}+\frac{1}{\sqrt{1+a^2}}\le\frac{C}{a}=\frac{C}{n-1+d}.
\end{align*}

Thus, from \eqref{eq:ATD-crossing-repr} we have
\begin{align*}
\E\abs{\cI_{[0,M]}}&=\sum_{n=1}^{M}\pr(Y_{U,n-1}Y_{U,n}<0)\le C\sum_{n=1}^{M}\frac{1}{n-1+d}\le
2C\sum_{k=d}^{d+M-1}\ln\left(1+\frac1k\right)=2C\ln\left(1+\frac{M}{d}\right),
\end{align*}
where in the second inequality we used that $1/k\le2\ln\left(1+1/k\right)$ for every $k\ge1$, allowing us to convert the sum to the telescoping sum of logarithms.
\end{proof}

\subsection{Lower bound in Theorem~\ref{thm:ATDwalk}}\label{sec:ATD_lower_bdd}
The main aim of this subsection is to prove the following proposition, that repeats the lower bound from Theorem~\ref{thm:ATDwalk}.
\begin{prop}\label{prop:ATD_lower_bdd}
    In the situation of Theorem~\ref{thm:ATDwalk}, there is a universal constant $c>0$ such that 
    \[
    \E\abs{\cI_{[0,M]}}\ge  c\ln\left(1+\frac{M}{d}\right).
    \]
\end{prop}
To derive this lower bound we will need two lemmas. The first one gives a uniform lower bound for the expected size of the interface.
\begin{lem}\label{lem:uniform_lower_bd}
In the situation of Theorem~\ref{thm:ATDwalk}, there is a universal constant $c>0$ 
\begin{equation}
        \E\abs{\cI_{[0,M]}}\ge c.
    \end{equation}
\end{lem}
\begin{proof}
    By Lemma~\ref{lem:tau-tail} and the fact that $\tau^{(2)}_d\stackrel{\mathrm d}{=}\tau_{2d}$ we have that for every sufficiently large universal constant $K$
    $$
    \pr\left(\tau^{(2)}_d>Kd^2\right)\le\frac{C}{\sqrt K}\le\frac12.
    $$
    On the other hand, using the same trick as before and writing $\tau^{(1)}_d\stackrel{\mathrm d}{=}\eta_1+\cdots+\eta_d,$ where $\eta_1,\ldots,\eta_d$ are independent copies of $\tau_1$ under $\pr_0$, Lemma~\ref{lem:basic_bounds}.(1) gives 
    \begin{align*}
    \pr\left(\tau^{(1)}_d>Kd^2\right)&\ge\pr\left(\max_{1\le k\le d}\eta_k>Kd^2\right)=1-\left(1-\pr(\eta_1>Kd^2)\right)^d\ge1-\exp\left(-\frac{C}{\sqrt K}\right)>0
    \end{align*}
    for some $C>0$. By independence of $\tau^{(1)}_d$ and $\tau^{(2)}_d$, we have
    $$
    \pr\left(\tau^{(2)}_d\le Kd^2<\tau^{(1)}_d\right)
    \ge \frac{1}{2}\left(1-\exp\left(-\frac{C}{\sqrt K}\right)\right):=c>0.
    $$
    On this event, the vertex $d\in[0,M]$ is claimed by $X^{(2)}_t$ and thus there must be at least one interface edge in the interval. 
\end{proof}

The second lemma provides a uniform lower bound for the probability of a zero-crossing of the ATD walk in an interval of a given length; its proof is placed at the end of the section.

\begin{lem}\label{lem:oneblock-crossing}
Let $Y_{u,n}$ be the shifted ATD walk as in \eqref{eq:shifted ATD walk} with deterministic shift $u\in\bR$. Then there is a universal constant $p>0$ such that, for every $m\ge1$ and $0\le u\le m^2$,
\[
\pr\left(\exists n\in\{m+1,\ldots,2m\}:Y_{u,n-1}Y_{u,n}\le0\right)\ge p.
\]
\end{lem}

We now give a proof of Proposition~\ref{prop:ATD_lower_bdd}, which is together with Proposition~\ref{prop:ATD_upper_bdd} complete the proof of Theorem~\ref{thm:ATDwalk}.

\begin{proof}[Proof of Proposition~\ref{prop:ATD_lower_bdd}]
    By \eqref{eq:ATD-crossing-repr}, it suffices to lower bound the expected number of zero-crossing of the ATD walk $Y_{U,n}$ with random shift $U=\tau_0^{(2)}$. 
    For convenience, we first consider deterministic shift $u\ge0$ and show that there are constants $c_1,C_1>0$ such that
    \begin{equation}\label{eq:determ_ATD_lower_bdd}
        \E\left(\sum_{n=1}^{M}\ind_{\{Y_{u,n-1}Y_{u,n}\le0\}}\right)\ge c_1\ln\left(1+\frac{M}{\sqrt{1+u}}\right)-C_1.
    \end{equation}

    For ease of notation, denote the sum inside of the expectation by $W=W(M,u)$. Set $m_0 = \lceil \sqrt{1+u}\rceil$, for $j\in\bZ_+$ define $m_j = 2^jm_0$, and let $\cJ:=\left\{j\in\bZ_+: m_j\le M\right\}.$  Define an event
    \[
    A_j = \{\exists n\in\{m_{j-1}+1,\ldots,m_j\}:Y_{u,n-1}Y_{u,n}\le0\}.
    \]
    By Lemma~\ref{lem:oneblock-crossing}, there is a universal $p > 0$ such that $\pr(A_j) \ge p$ uniformly for all $j$. 
    
    We then consider two cases, depending on whether $M\ge 2m_0$ (and thus interval $[0,M]$ contains at least one block from an event $A_j$) or not. 
    If $M\ge2m_0$, then $\abs{\cJ}= \lfloor \log_2(M/m_0) \rfloor\ge1$ represents the number of such blocks inside $[0,M]$. Since each event $A_j$ ensures at least 1 crossing, we have that 
    \[
    \E(W)\ge \sum_{j=1}^{\abs{\cJ}}\pr(A_j)\ge p\abs{\cJ}.
    \]
    Since $m_0 = \lceil \sqrt{1+u} \rceil$, this yields 
    \begin{equation}\label{eq:EW_part1}
    \E(W) > p\left(\log_2\left(\frac{M}{m_0}\right)-1\right)\ge \frac{p}{\ln(2)} \ln\left( 1 + \frac{M}{\sqrt{1+u}} \right) - 3p,
    \end{equation}
    where we use the fact that $\ln(x) \ge\ln(1+x)-\ln(2)$ for every $x\ge 1$.

    In the case when $M<2m_0$, the desired inequality holds for trivial reasons; indeed $\E(W) \ge 0$ while right-hand side of \eqref{eq:determ_ATD_lower_bdd} is at most $-p<0$ as $M\le2m_0-1$ yields that
    \begin{equation*}
    \frac{p}{\ln(2)} \ln\left( 1 + \frac{2m_0-1}{\sqrt{1+u}} \right) - 3p<\frac{p}{\ln(2)} \ln\left( 3 + \frac{1}{\sqrt{1+u}} \right) - 3p\le\frac{p}{\ln(2)} \ln(4) - 3p=-p\le0.
    \end{equation*}
    This, together with \eqref{eq:EW_part1} gives \eqref{eq:determ_ATD_lower_bdd}.

    We now return to the random shift $U=\tau^{(2)}_0.$ By Lemma~\ref{lem:tau-tail} for some universal constant $C>0$ and every $x\ge1$ the following inequality holds
    $$
    \pr(U>x)\le\frac{Cd}{\sqrt{x}}.
    $$  
    Choosing a constant $k$ large enough and setting $x=kd^2$ yields $\pr(U\le kd^2)\ge 1/2.$
    By independence of $U$ and $(S_n)_{n\ge0}$, conditioning on the former and using \eqref{eq:determ_ATD_lower_bdd} we derive
    \begin{align*}
    \E\abs{\cI_{[0,M]}}&=\E\left(\E\left(\left.\sum_{n=1}^{M}\ind_{\{Y_{U,n-1}Y_{U,n}\le0\}}\right|U\right)\right)\\
    &\ge\pr(U\le kd^2)\left(c_1\ln\left(1+\frac{M}{\sqrt{1+kd^2}}\right)-C_1\right)\\
    &\ge c_2\ln\left(1+\frac{M}{d}\right)-C_2.
    \end{align*}
    where in the first inequality we used that the conditional expectation is nonnegative and that the right-hand side of \eqref{eq:determ_ATD_lower_bdd} is decreasing in $u$; in last inequality we used that $d\ge1$ (to get $\sqrt{1+kd^2}\le d\sqrt{1+k}$) and $M/d>1$ to simplify the expression.
    
    Finally to remove the negative constant in the inequality above we consider two cases.
    If $\ln\left(1+M/d\right)\ge {2C_2}/{c_2}$, then 
    $$
    \E\abs{\cI_{[0,M]}}\ge\frac{c_2}{2}\ln\left(1+\frac Md\right).
    $$

    Otherwise, using Lemma~\ref{lem:uniform_lower_bd} we get
    $$
    \E\abs{\cI_{[0,M]}}\ge\frac{c\cdot c_2}{2C_2}\ln\left(1+\frac Md\right).
    $$
    Choosing the smallest of the two constants completes the proof.
\end{proof}

It remains to prove Lemma~\ref{lem:oneblock-crossing}. To do that we first provide a uniform bound on the sizes of the jumps of the ATD walk and their cancellations. We introduced a notation for the walk without $j$-th step, for some $1\le j\le n$, define
\[
   S_n^{(j)}:=\sum_{\substack{1\le i\le n\\ i\ne j}}\zeta_i.
\]

\begin{lem}\label{lem:jump sizes}
Fix $k,m\ge1$. Let $S_n$ be the ATD walk as in~\eqref{eq:ATD walk}. Then there are universal constants $C_1$ and $C_2$ such that 
\[
  \pr\left(\exists\,1\le j\le m:k m^2<\zeta_j\le 2k m^2\right)  \ge \frac{C_1}{\sqrt{k}}
\qquad\mbox{and}\qquad
  \pr\left(\abs{ S_m^{(j)}}>\frac12 km^2\right)\le \frac{C_2}{\sqrt{k}},
\]
for each $1\le j\le m$.
\end{lem}

\begin{proof}
Define an event $B_j:=\{km^2<\zeta_j\le 2k m^2\}$.
Set $x = km^2$ and let $\eta$ and $\eta'$ be as in Lemma~\ref{lem:basic_bounds}, recalling that $\zeta$ has the same distribution as $\eta-\eta'$. Observe that
\[
\left\{\frac{3}{2}x<\eta\le 2x,\ \eta'\le \frac{x}{2}\right\}\subseteq\{x<\eta-\eta'\le2x\}.
\]
Thus, by independence and Lemma~\ref{lem:basic_bounds}(2), which applies because $x\ge1$,
\begin{equation}\label{eq:pr(Bj)}
p:=\pr(B_j)\ge\pr\left(\frac{3}{2}x<\eta\le 2x\right)\pr\left(\eta'\le \frac{x}{2}\right)\ge \frac{c}{ \sqrt{x}}\pr\left(\eta'\le\frac12\right)=\frac{c_1}{m\sqrt{k}},
\end{equation}
where $c_1>0$ is universal. Since the $(\zeta_j)_{j\ge1}$ are independent, the probability that at least one falls into desired $[km^2,2km^2]$ is $1-(1-p)^m$. The first desired inequality then follows from
\begin{align*}
1-(1-p)^m &\ge 1-\left(1-\frac{c_1}{m\sqrt{k}}\right)^m\ge 1-e^{-c_1/\sqrt{k}}\ge \frac{1-e^{-c_1}}{\sqrt{k}}.
\end{align*}

To derive the second desired inequality, we observe that
\begin{equation}\label{eq:superevents}
    \left\{\abs{ S_m^{(j)}} > \frac{x}{2}\right\} \subseteq \left\{\abs{\sum_{\substack{i\ne j}}\zeta_i\ind_{\{|\zeta_i|\leq x\}}} > \frac{x}{2}\right\} \cup \left\{\exists i \ne j : \abs{\zeta_i}>x\right\}
\end{equation}
This allows us to find separate bounds for each event. For the first, note that each $\zeta_i\ind_{\{|\zeta_i|\leq x\}}$ has mean $0$, so by Lemma~\ref{lem:basic_bounds}.(4) and Chebyshev's inequality, we can fix a constant $c_1$ such that
\begin{equation}\label{eq:event1}
    \pr\left(\abs{\sum_{i\ne j}\zeta_i\ind_{\{|\zeta_i|\leq x\}}} > \frac{x}{2}\right) \le \frac{4\var\left(\sum_{i\ne j}\zeta_i\ind_{\{|\zeta_i|\leq x\}}\right)}{x^2} \le \frac{(m-1)c_1x^{3/2}}{x^2} \le \frac{mc_1}{\sqrt{x}}=\frac{c_1}{\sqrt{k}}
\end{equation}

We bound the remaining term again by Lemma~\ref{lem:basic_bounds}.(3). Since $\zeta_i$ have a symmetric distribution there is a constant $c_2$ such that 
\[
\pr\left(|\zeta_i| > x\right) = 2\pr\left(\zeta_i>x\right) \le \frac{c_2}{\sqrt{x}}.
\]
Using the union bound, the probability of the second event on the right-hand side of \eqref{eq:superevents} 
\begin{equation}\label{eq:event2}
\pr\left(\exists i \ne j \text{ s.t. } |\zeta_i|> x\right) \le m\pr\left(|\zeta_i|> x\right) \le \frac{mc_2}{\sqrt{x}}=\frac{c_2}{\sqrt{k}}.
\end{equation}
Combining \eqref{eq:superevents} with inequalities from \eqref{eq:event1} and \eqref{eq:event2} concludes the proof.
\end{proof}

\begin{lem}\label{lem:2_interval_lemma}
Let $S_n$ be the ATD walk as in~\eqref{eq:ATD walk}. Then there are constants $q_1,q_2>0$ and intervals $[x_1, x_2], [y_1, y_2]$ with $y_1 > x_2+1$ such that, for every integer $m\ge1$,
\[
  \pr\left(-x_2m^2\le  S_m\le -x_1m^2\right)\ge q_1
\qquad\mbox{and}\qquad
  \pr\left(y_1m^2\le  S_m\le y_2m^2\right)\ge q_2.
\]
\end{lem}

\begin{proof}
Fix $k\ge1$ that will be specified later. For $1\le j\le m$, define 
\[
B_j:=\{km^2<\zeta_j\le2km^2\} \quad\mbox{and}\quad E_j:=\left\{-\frac12 km^2\le S_m^{(j)}\le \frac12 km^2\right\}.
\]
Since $S_m=\zeta_j+S_m^{(j)},$ we have
\[
  \bigcup_{j=1}^m(B_j\cap E_j)
  \subseteq
  \left\{\frac12 km^2<S_m\le \frac{5}{2}km^2\right\}.
\]

To lower-bound the probability of the union on the left-hand side we use the first two terms from the inclusion-exclusion formula

\begin{equation}\label{eq:incl-excl}
  \pr\left(\bigcup_{j=1}^m(B_j\cap E_j)\right)\ge\sum_{j=1}^m\pr(B_j\cap E_j)-\sum_{1\le i<j\le m}\pr(B_i\cap E_i\cap B_j\cap E_j)
\end{equation}

To bound the first term in \eqref{eq:incl-excl}, recall that for each
fixed $j$ by \eqref{eq:pr(Bj)} we have $\pr(B_j)\ge{c_1}/{m\sqrt{k}}.$ Since $B_j$ and $E_j$ are independent, Lemma~\ref{lem:jump sizes} gives that
\[
  \pr(E_j^c)=\pr\left(|S_m^{(j)}|>\frac12 km^2\right)\le\frac{c_2}{\sqrt{k}}.
\]
Choose $k$ large enough that $\sqrt{k}\ge 2c_2$, then $\pr(E_j)\ge1/2,$ and therefore
\[
  \pr(B_j\cap E_j) = \pr(B_j)\pr(E_j)  \ge  \frac{c_3}{m\sqrt{k}}.
\]
Summing over $j$, we get
\[
  \sum_{j=1}^m\pr(B_j\cap E_j)  \ge  \frac{c_3}{\sqrt{k}}.
\]

Next, we upper bound the second term in \eqref{eq:incl-excl}. Note that by Lemma~\ref{lem:basic_bounds}.(3)
\[
  \pr(B_j)\le \pr(\zeta_j>km^2)
  \le
  \frac{c_4}{m\sqrt{k}}.
\]
Since $B_i$ and $B_j$ are independent for $i\ne j$, we then have 
$$
\pr(B_i\cap B_j)=\pr(B_i)\pr(B_j)\le\frac{c_5}{m^2k}.
$$
Therefore
\[
    \sum_{1\le i<j\le m}\pr(B_i\cap E_i\cap B_j\cap E_j)\le\sum_{1\le i<j\le m}\pr(B_i\cap B_j)\le\frac{c_6}{k}.
\]

Recalling \eqref{eq:incl-excl} and choosing $k$ large enough so that $\sqrt{k}\ge 2c_6/c_3$ we get
\begin{equation}\label{eq:main_bound}
  \pr\left(\frac12 km^2<S_m\le \frac{5}{2}km^2\right)\ge\pr\left(\bigcup_{j=1}^mB_j\cap E_j\right)\ge  \frac{c_3}{\sqrt{k}}-\frac{c_6}{k}\ge \frac{c_3}{2\sqrt{k}}=:q_k>0.
\end{equation}
By symmetry of $S_m$ we conclude
\[
  \pr\left(-\frac{5}{2} km^2<S_m\le -\frac12 km^2\right)\ge q_k>0.
\]
Choose $k_0$ large enough that the second inequality is satisfied, and set $x_1:=\frac12 k_0$,  $x_2:=\frac{5}{2}k_0,$  and  $q_1:=q_{k_0}$, This gives the desired first inequality. Next choose $k_1$ sufficiently large so that $ k_1> 2(x_2+1)=5k_0+2.$
Applying \eqref{eq:main_bound} with $k_1$ and setting  $y_1:=\frac12 k_1$,  $y_2:=\frac{5}{2}k_1,$ and $q_2:=q_{k_1}$ completes the proof.
\end{proof}

We are finally ready to give a proof of Lemma~\ref{lem:oneblock-crossing} completing the section.

\begin{proof}[Proof of Lemma~\ref{lem:oneblock-crossing}]
    Let the constants $x_1,x_2,y_1,y_2,q_1,q_2$ be given by Lemma~\ref{lem:2_interval_lemma}. Fix $m\ge1$ and $u$ satisfying $0\le u\le m^2$ and consider the two events
    \[
    A:=\left\{-x_2m^2\le  S_m\le -x_1m^2\right\}\qquad\mbox{and}\qquad
    B:=\left\{y_1m^2\le  S_{2m}- S_m\le y_2m^2\right\}.
    \]
    The random variable $ S_{2m}- S_m$ has the same distribution as $ S_m$ and is independent of from it. By Lemma~\ref{lem:2_interval_lemma} we have $\pr(A)\ge q_1$ and $\pr(B)\ge q_2$ and hence
    \[
    \pr(A\cap B)=\pr(A)\pr(B)\ge q_1q_2.
    \]

    We now show that on $A\cap B$, the sequence $(Y_{u,n})$ crosses $0$ at least once between $m+1$ and $2m$. On $A$, we have
    \[
        Y_{u,m}= S_m-u\le -x_1m^2-u<0,
    \]
    while on $A\cap B$ we have
    \[
         S_{2m}= S_m+( S_{2m}- S_m)\ge-x_2m^2+y_1m^2=(y_1-x_2)m^2.
    \]
    Since $u\le m^2$ and $x_2+1<y_1$ we derive that
    \[
        Y_{u,2m}= S_{2m}-u\ge(y_1-x_2)m^2-m^2=(y_1-x_2-1)m^2>0.
    \]
    Thus, the desired event contains $A\cap B$ and is also of probability at least $p:=q_1q_2>0$.
\end{proof}

\section{Proof of Theorem~\ref{thm:main}}\label{sec:main-proof}
As before let $X^{(1)}_t$ and $X^{(2)}_t$ be two independent continuous time random walks on $\bZ$ started at started at $0$ and $-d$, respectively.
Assume $d\le N/2$ and identify the cycle $C_N$ with $\bZ/N\bZ$.
Then considering these walks $\mathrm{mod}\,N$, gives two independent continuous time random walks $W^{(1)}_t$ and $W^{(2)}_t$ on $C_N$ started at $0$ and $N-d$, respectively.
Decompose $C_N$ into two arcs, setting $L:=N-d$ for convenience
\[
\cA_{\rm long}=C_N|_{\{0,1,\ldots,L\}} \qquad\mbox{and}\qquad \cA_{\rm short}=C_N|_{\{0,-1,\ldots,-d\}}.
\]

To establish Theorem~\ref{thm:main} we will focus on bounding the expected size of the interface in $\cA_{\rm long}$. 
Clearly, that is enough for the lower bound. On the other hand, for the upper bound we will rely on the fact that Proposition~\ref{prop:ATD_upper_bdd} holds even when $M<d$, making the short arc behave analogously (setting $d'=N-d$ and $M'=d$). We do so primarily to make the reduction to the setting of Section~\ref{sec:ATD-walk} more clear, avoid repetition and introducing even more convoluted notation.

Recalling notation of hitting times $\tau^{(i)}_n$ of a vertex $n$ by $X^{(i)}_t$ we define the induced hitting for $W^{(i)}_t$ on $\cA_{\rm long}$ by setting
$$
\gs_n^{(i)}:=\min\left\{\tau^{(i)}_n,\tau^{(i)}_{n-N}\right\},
$$
for each $n$ such that $W^{(i)}_0\in[n-N,n]$ (seen as the arc of $C_N$ that contains $0$).
As before note that the random variables $\gs_n^{(i)}$, for each $n\in C_N$ and $i\in\{1,2\}$, are almost sure distinct.

The upper bound in Theorem~\ref{thm:main} follows directly from Proposition~\ref{prop:ATD_upper_bdd} after decomposition of the cycle into three (random) intervals.

\begin{prop}\label{prop:cycle_upper_bound}
In the situation of Theorem~\ref{thm:main}. There is a universal constant $C<\infty$ such that
\[
\E\abs{\cI_{C_N}}\le C\ln\left(1+\frac{N}{d}\right).
\]
\end{prop}
\begin{proof}

First we rewrite the size of the interface $\cI_{\cA_{\rm long}}$ via arrival times similar to \eqref{eq:ATD-crossing-repr}
$$
\abs{\cI_{\cA_{\rm long}}}=\sum_{n=0}^{L-1}\ind_{\left\{(\gs_n^{(1)}-\gs_n^{(2)})(\gs_{n+1}^{(1)}-\gs_{n+1}^{(2)})<0\right\}}.
$$

For each $i\in\{1,2\}$, the sequence
$
\tau_n^{(i)}
$
is nondecreasing in $n\in\{0,1,\ldots, L\}$, while
$
\tau_{n-N}^{(i)}
$
is nonincreasing in $n\in\{0,1,\ldots, L\}$. Thus, as $n$ moves along
$\cA_{\rm long}$, the minimum defining $\gs_n^{(i)}$ can change from
$\tau_n^{(i)}$ to $\tau_{n-N}^{(i)}$ at most once.

Hence, the two walks therefore partition $\cA_{\rm long}$ into three (possibly empty) intervals:
\begin{itemize}
    \item[$I_1$:] both $\gs_n^{(i)}$ are given by $\tau_n^{(i)}$;
    \item[$I_2$:] both $\gs_n^{(i)}$ are given by $\tau_{n-N}^{(i)}$;
    \item[$I_3$:] $\gs_n^{(1)}$ and $\gs_n^{(2)}$ are given by $\tau_n^{(i)}$ and $\tau_{n-N}^{(j)}$ for the appropriate pairing $i,j\in\{1,2\}$.
\end{itemize}

Note that in the last case, the expression $\tau_n^{(1)}-\tau_{n-N}^{(2)}$ is nondecreasing in $n$, while $\tau_{n-N}^{(1)}-\tau_n^{(2)}$ is nonincreasing. 
In particular, on $I_3$, the $\gs_n^{(1)}-\gs_n^{(2)}$ can change sign at most once, creating at most one edge in $\cI_{\cA_{\rm long}}$.
Accounting for possible interface edges adjacent at the places where the induced ATD walk crosses zero we obtain
\begin{align*}
\abs{\cI_{\cA_{\rm long}}}&\le 6+\sum_{n=0}^{L-1}\ind_{\left\{(\tau_n^{(1)}-\tau_n^{(2)})(\tau_{n+1}^{(1)}-\tau_{n+1}^{(2)})<0\right\}}+\sum_{n=0}^{L-1}\ind_{\left\{(\tau_{n-N}^{(1)}-\tau_{n-N}^{(2)})(\tau_{n+1-N}^{(1)}-\tau_{n+1-N}^{(2)})<0\right\}}.
\end{align*}

Note that by symmetry the sums both sums have the same distribution, which also coincides with the distribution of the interface $\cI_{[0,L]}$ for the competition on $\bZ$ considered in Section~\ref{sec:ATD-walk}. Thus, applying Proposition~\ref{prop:ATD_upper_bdd} gives
$$
\E\abs{\cI_{\cA_{\rm long}}}\le C\left(1+\ln\left(1+\frac Ld\right)\right).
$$

Since Proposition~\ref{prop:ATD_upper_bdd} does not rely on $d<M$, the exact same argument for $\cI_{\cA_{\rm short}}$ yields
$$
\E\abs{\cI_{\cA_{\rm short}}}\le C\left(1+\ln\left(1+\frac{d}{L}\right)\right)\le C.
$$
Finally, using that $L=N-d$ and $N/d\ge2$, changing the constant gives the desired upper bound
$$
\E\abs{\cI_{C_N}}\le C\ln\left(1+\frac{N}{d}\right),
$$
for some constant $C>0$.
\end{proof}

As we mentioned before the lower bound does not follow directly from Proposition~\ref{prop:ATD_lower_bdd} and instead our argument will follows similar steps. The next lemma compares crossings of the ATD walk with interfaces on the long arc of the cycle.  The only possible discrepancy is that one of the
$W^{(i)}_t$ reaches a vertex through the other direction of the cycle.

\begin{lem}\label{lem:cycle_block}
For $n\ge0$ and $d\ge1$, let $Y_{U,n}$ be the shifted ATD walk as in \eqref{eq:shifted ATD walk} with $U=\tau^{(2)}_{0}$.
Suppose that $m\ge1$ and $2m<N-d.$ Then
\begin{align*}
&\pr\left(\exists n\in\{m+1,\ldots,2m\}:(n-1,n)\in\cI_{\cA_{\rm long}}\right)\\
&\qquad\ge\pr\left(\exists n\in\{m+1,\ldots,2m\}:Y_{U,n-1}Y_{U,n}<0\right)-\frac{4m+d}{N}.
\end{align*}
\end{lem}

\begin{proof}
Define
\[
G_m^{(1)}:=\left\{\tau_{2m}^{(1)}<\tau_{2m-N}^{(1)}\right\},
\qquad
G_m^{(2)}:=\left\{\tau_{2m}^{(2)}<\tau_{2m-N}^{(2)}\right\},
\]
and set $G_m:=G_m^{(1)}\cap G_m^{(2)}.$
Since $2m<N-d$, both starting points $0$ and $-d$ belong to the interval $[2m-N,2m].$
On $G_m^{(1)}$, for every $m\le n\le2m$, we have
\[
\tau_n^{(1)}\le\tau_{2m}^{(1)}<\tau_{2m-N}^{(1)}\le\tau_{n-N}^{(1)}.
\]
Similarly, on $G_m^{(2)}$, analogous inequality hold
\[
\tau_n^{(2)}\le\tau_{2m}^{(2)}<\tau_{2m-N}^{(2)}\le\tau_{n-N}^{(2)}.
\]
In particular, on $G_m$ we have that $\gs_n^{(i)}=\tau_n^{(i)},$ for $m\le n\le2m$ and $i\in\{1,2\}.$ Therefore, on $G_m$, every crossing of zero by the ATD walk $Y_{U,n}$ between $n$ and $n+1$, for $m\le n<2m$, produces an interface edge $\{n-1,n\}$ on the long arc.
It follows that
\begin{align*}
&\pr\left(\exists n\in\{m+1,\ldots,2m\}:(n-1,n)\in\cI_{\cA_{\rm long}}\right)\\
&\qquad\qquad\ge\pr\left(\left\{\exists n\in\{m+1,\ldots,2m\}:Y_{U,n-1}Y_{U,n}<0\right\}\cap G_m\right)\\
&\qquad\qquad\ge\pr\left(\exists n\in\{m+1,\ldots,2m\}:Y_{U,n-1}Y_{U,n}<0\right)-\pr(G_m^c).
\end{align*}
It remains to show that $\pr(G_m^c)\le{(4m+d)}/{N}$. This follows directly from the classical gambler's ruin problem; indeed we have $\pr((G_m^{(1)})^c)={2m}/{N}$ and $\pr((G_m^{(2)})^c)=(2m+d)/{N}$, using that $X^{(2)}_0=-d$.
\end{proof}

\begin{proof}[Proof of Theorem~\ref{thm:main}]
Since the upper bound is given in Proposition~\ref{prop:cycle_upper_bound}, it remains to prove the lower bound.

Let $p>0$ be the universal constant given by Lemma~\ref{lem:oneblock-crossing}.  
Using Lemma~\ref{lem:tau-tail}, choose a universal constant $A>1$ sufficiently large that whenever $m\ge Ad$ we have
\begin{equation*}
    \pr(U\le m^2)\ge\frac12.
\end{equation*}
Here we enforce that $A>1$ for convince to guarantee that $m>d$.

Using similar conditioning trick as in the proof of Proposition~\ref{prop:ATD_lower_bdd}, conditioning on $U$ and assuming that $m\ge Ad$, we derive the following.  By Lemma~\ref{lem:oneblock-crossing} and the independence of $U$ and $( S_n)_{n\ge1}$, for every $u\le m^2$ we get
\begin{align*}
&\pr\left(\exists n\in\{m+1,\ldots,2m\}:Y_{U,n-1}Y_{U,n}<0\right)\\
&\qquad\ge\E\left(\ind_{\{U\le m^2\}}\pr\left(\left.\exists n\in\{m+1,\ldots,2m\}:Y_{U,n-1}Y_{U,n}<0\right|U\right)\right)\ge p\pr(U\le m^2)\ge\frac{p}{2}.
\end{align*}

We now choose a universal constant $\eps\in (0,{p}/{10})$ and further assume $2m\le\eps N.$ Note this together with $m\ge Ad> d$ implies that $d<N/2$ and thus $2m<N-d$. Lemma~\ref{lem:cycle_block} then gives for every $Ad\le m\le\eps N/2$
\begin{equation}\label{eq:cycle_block_probability}
\pr\left(\exists n\in\{m+1,\ldots,2m\}:(n-1,n)\in\cI_{\cA_{\rm long}}\right)\ge\frac{p}{2}-\frac{4m+d}{N}\ge\frac{p}{4},
\end{equation}
where in the last inequality we used that $\eps<p/10$, $2m\le\eps N$, and $m> d$ to get
\[
\frac{4m+d}{N}\le\frac{5m}{N}\le\frac52\eps\le \frac{p}{4}.
\]

The rest of the proof is similar to the proof of Proposition~\ref{prop:ATD_lower_bdd}, we apply this bound on a sequence of exponentially growing blocks. 

First, $m_0:=\lceil Ad\rceil$, for $j\in\bZ_+$ define $m_j:=2^jm_0,$ and let
\[
\cJ:=\left\{j\in\bZ_+:m_j\le\eps N\right\}.
\]
Since the sequence of blocks $\{m_{j=1}+1,\ldots,2m_j\}$ is edge-disjoint, by applying \eqref{eq:cycle_block_probability}, we obtain
\begin{equation*}
    \E\abs{\cI_{C_N}}
\ge\E\left(\sum_{j\in\cJ}\ind_{\left\{\exists n\in\{m_{j-1},\ldots,2m_j\}:(n-1,n)\in\cI_{\cA_{\rm long}}\right\}}\right)\ge \frac{p}{4}\abs{\cJ}.
\end{equation*}

Since $m_0\le(A+1)d$, the number of the blocks in the sequence is at most
\[
\abs{\cJ}\ge c_1\ln\left(\frac Nd\right)-C_1
\]
for universal constants $c_1,C_1>0$.  Therefore
\[
\E\abs{\cI_{C_N}}
\ge
c_2\ln\left(\frac Nd\right)-C_2.
\]

In case of $\cI_{C_N}$ removing the negative constant is from the expression above is actually easier than for $\cI_{[0,M]}$, where we had to rely on the uniform bound from Lemma~\ref{lem:uniform_lower_bd}. Indeed, since the initial positions of $W^{(1)}_t$ and $W^{(2)}_t$ are distinct, the interface on the cycle automatically contains at least $2$ edges almost surly. After changing the universal constant, we conclude that $\E\abs{\cI_{C_N}}\ge c\ln\left(1+N/d\right).$
\end{proof}

\bibliographystyle{abbrvurl}
\bibliography{rw} 
\end{document}

%% file: style.tex
\usepackage{mathrsfs,xfrac} 
\usepackage[colorlinks=true,linkcolor=blue,citecolor=blue,urlcolor=blue]{hyperref} 
\usepackage{amsmath,amssymb,amsthm,amsfonts,amsbsy,latexsym,dsfont,color,graphicx,enumitem}
\usepackage{caption}
\usepackage{regexpatch}
\usepackage{comment}
\makeatletter
\usepackage{todonotes}
\xpatchcmd{\@todo}{\setkeys{todonotes}{#1}}{\setkeys{todonotes}{inline,#1}}{}{}
\makeatother

\newtheorem{thm}{Theorem}[section]
\newtheorem{lem}[thm]{Lemma}

\newtheorem{prop}[thm]{Proposition}

\renewcommand{\le}{\leqslant}  
\renewcommand{\ge}{\geqslant}

\newcommand{\ind}{\mathds{1}}
\newcommand{\eps}{\varepsilon}

\newcommand{\abs}[1]{\left\vert#1\right\vert}

              \let\gs=\sigma

\newcommand{\cA}{\mathcal{A}}

\newcommand{\cI}{\mathcal{I}}
\newcommand{\cJ}{\mathcal{J}}

\newcommand{\bC}{\mathbb{C}}

\newcommand{\bN}{\mathbb{N}}
\newcommand{\bR}{\mathbb{R}}

\newcommand{\bZ}{\mathbb{Z}}        

\DeclareMathOperator{\E}{\mathds{E}}
\DeclareMathOperator{\pr}{\mathds{P}}

\DeclareMathOperator{\var}{Var}

\newcommand{\wh}[1]{\widehat{#1}}


%% file: rw.bib
@book{feller1991introduction,
  title={An introduction to probability theory and its applications},
  author={Feller, W.},
  volume={2},
  year={1991},
  publisher={John Wiley \& Sons}
}

@article {CollabRW,
    AUTHOR = {Dey, P.~S. and Kim, D. and Terlov, G.},
     TITLE = {Collaboration of random walks on graphs},
   JOURNAL = {Proc. Amer. Math. Soc.},
  FJOURNAL = {Proceedings of the American Mathematical Society},
    VOLUME = {153},
      YEAR = {2025},
    NUMBER = {11},
     PAGES = {4961--4974},
      ISSN = {0002-9939,1088-6826},
   MRCLASS = {60G50 (05C81 60B15 60F99)},
  MRNUMBER = {4971586},
MRREVIEWER = {Michael\ Voit},
       DOI = {10.1090/proc/17332},
}

@article {Hermon25,
    AUTHOR = {Hermon, J. and Sousi, P.},
     TITLE = {Covering a graph with independent walks},
   JOURNAL = {Electron. J. Probab.},
  FJOURNAL = {Electronic Journal of Probability},
    VOLUME = {30},
      YEAR = {2025},
     PAGES = {Paper No. 178, 24},
      ISSN = {1083-6489},
   MRCLASS = {60J10 (60J27)},
  MRNUMBER = {4987417},
MRREVIEWER = {Yoshihiro\ Abe},
       DOI = {10.1214/25-ejp1348},
}

@article{rivera2023multiple,
  title={Multiple random walks on graphs: Mixing few to cover many},
  author={Rivera, N. and Sauerwald, T. and Sylvester, J.},
  journal={Combinatorics, Probability and Computing},
  volume={32},
  number={4},
  pages={594--637},
  year={2023},
  publisher={Cambridge University Press}
}

@inproceedings{sauerwald2010expansion,
  title={Expansion and the cover time of parallel random walks},
  author={Sauerwald, T.},
  booktitle={Proceedings of the 29th ACM SIGACT-SIGOPS symposium on Principles of distributed computing},
  pages={315--324},
  year={2010}
}

@article{patel2016hitting,
  title={The hitting time of multiple random walks},
  author={Patel, R. and Carron, A. and Bullo, F.},
  journal={SIAM Journal on Matrix Analysis and Applications},
  volume={37},
  number={3},
  pages={933--954},
  year={2016},
  publisher={SIAM}
}

@inproceedings{ivaskovic2017multiple,
  title={Multiple random walks on paths and grids},
  author={Ivaskovic, A. and Kosowski, A. and Pajak, D. and Sauerwald, T.},
  booktitle={34th Symposium on Theoretical Aspects of Computer Science (STACS 2017)},
  year={2017},
  organization={Schloss-Dagstuhl-Leibniz Zentrum f{\"u}r Informatik}
}

@article{elsasser2011tight,
  title={Tight bounds for the cover time of multiple random walks},
  author={Els{\"a}sser, R. and Sauerwald, T.},
  journal={Theoretical Computer Science},
  volume={412},
  number={24},
  pages={2623--2641},
  year={2011},
  publisher={Elsevier}
}

@inproceedings{efremenko2009well,
  title={How well do random walks parallelize?},
  author={Efremenko, K. and Reingold, O.},
  booktitle={International Workshop on Approximation Algorithms for Combinatorial Optimization},
  pages={476--489},
  year={2009},
  organization={Springer}
}

@article{Alon_at_al, 
title={Many Random Walks Are Faster Than One}, 
volume={20}, DOI={10.1017/S0963548311000125}, 
number={4}, 
journal={Combinatorics, Probability and Computing},
author={Alon, N. and Avin, C. and Koucký, M. and Kozma, G. and Lotker, Z. and Tuttle, M. R.}, 
year={2011},
pages={481–502}
}

@article {DK21,
    AUTHOR = {Dey, P.~S. and Kim, D.},
     TITLE = {{F}luctuation results for size of the vacant set for random walks on discrete torus}, 
   JOURNAL = {ALEA Lat. Am. J. Probab. Math. Stat.},
  FJOURNAL = {ALEA. Latin American Journal of Probability and Mathematical Statistics},
    VOLUME = {22},
      YEAR = {2025},
     PAGES = {209–232},
       DOI = {10.30757/ALEA.v22-08},
}

@article{regnier2023universal,
  title={Universal exploration dynamics of random walks},
  author={R{\'e}gnier, L. and Dolgushev, M. and Redner, S. and B{\'e}nichou, O.},
  journal={Nature Communications},
  volume={14},
  number={1},
  pages={618},
  year={2023},
  publisher={Nature Publishing Group UK London},
}

@article{regnier2024maximum,
  title={From maximum of inter-visit times to starving random walks},
  author={R{\'e}gnier, L. and Dolgushev, M. and B{\'e}nichou, O.},
  journal={Physical Review Letters},
  volume={132},
  number={12},
  pages={127101},
  year={2024},
  publisher={APS},
}

@article{Bansaye24,
  title={Stochastic foraging paths primarily drives within-species variations of prey consumption rates},
  author={Bansaye, V. and Berthelot, G. and El Bachari, A. and Chazottes, J.-R. and Billiard, S.},
  journal={bioRxiv},
  year={2024},
  url={https://www.biorxiv.org/content/early/2024/06/02/2024.05.29.596370.}
}

@misc{Baccara26,
      title={On the range of competing random walks}, 
      author={Baccara, M.},
      year={2026},
      eprint={2607.02406},
      archivePrefix={arXiv},
      primaryClass={math.PR},
      url={https://arxiv.org/abs/2607.02406}, 
}

@article{Larralde92,
  title={Territory covered by N diffusing particles},
  author={Larralde, H. and Trunfio, P. and Havlin, S. and Stanley, H.~E. and Weiss, G.~H.},
  journal={Nature},
  volume={355},
  number={6359},
  pages={423--426},
  year={1992},
  publisher={Nature Publishing Group UK London}
}

@article{benichou2014depletion,
  title={Depletion-controlled starvation of a diffusing forager},
  author={B{\'e}nichou, O. and Redner, S.},
  journal={Physical Review Letters},
  volume={113},
  number={23},
  pages={238101},
  year={2014},
  publisher={APS}
}

@article {BCR09,
    AUTHOR = {Bass, R.~F. and Chen, X. and Rosen, J.},
     TITLE = {Moderate deviations for the range of planar random walks},
   JOURNAL = {Mem. Amer. Math. Soc.},
  FJOURNAL = {Memoirs of the American Mathematical Society},
    VOLUME = {198},
      YEAR = {2009},
    NUMBER = {929},
     PAGES = {viii+82},
      ISSN = {0065-9266,1947-6221},
      ISBN = {978-0-8218-4287-4},
   MRCLASS = {60F10 (60G50 60J55)},
  MRNUMBER = {2493313},
MRREVIEWER = {Endre\ Cs\'aki},
       DOI = {10.1090/memo/0929},
}

@mastersthesis {Dicker06,
	title={Coloring a d$\,\ge\,$3 dimensional lattice with two independent random walks},
  	author={Dicker, L.~},
  	year={2006},
  	school={Univ.~Pennsylvania},
	note={Available at \href{https://www.math.upenn.edu/~pemantle/papers/Student-theses/Masters/Dicker060421.pdf}{https://www.math.upenn.edu/$\sim$pemantle/papers/Student-theses/Masters/Dicker060421.pdf}},
}

@article {GLDSH96,
    AUTHOR = {Gomes Jr., S. R. and Lucena, L. S. and da Silva, L. R. and Hilhorst, H. J.},
     TITLE = {Coloring of a one-dimensional lattice by two independent random walkers},
   JOURNAL = {Phys. A},
  FJOURNAL = {Physica A},
    VOLUME = {225},
      YEAR = {1996},
    NUMBER = {1},
     PAGES = {81--88},
      ISSN = {0378-4371},
   MRCLASS = {82B41 (60J15)},
  MRNUMBER = {1379999},
MRREVIEWER = {Stuart G. Whittington},
       DOI = {10.1016/0378-4371(95)00424-6},
}

@article {Miller13,
    AUTHOR = {Miller, Jason},
     TITLE = {Painting a graph with competing random walks},
   JOURNAL = {Ann. Probab.},
  FJOURNAL = {The Annals of Probability},
    VOLUME = {41},
      YEAR = {2013},
    NUMBER = {2},
     PAGES = {636--670},
      ISSN = {0091-1798},
   MRCLASS = {60G50 (60F99)},
  MRNUMBER = {3077521},
MRREVIEWER = {Konstantin Borovkov},
       DOI = {10.1214/11-AOP713},
}

@article {TW11,
    AUTHOR = {Teixeira, Augusto and Windisch, David},
     TITLE = {On the fragmentation of a torus by random walk},
   JOURNAL = {Comm. Pure Appl. Math.},
  FJOURNAL = {Communications on Pure and Applied Mathematics},
    VOLUME = {64},
      YEAR = {2011},
    NUMBER = {12},
     PAGES = {1599--1646},
      ISSN = {0010-3640},
   MRCLASS = {05C81 (05D40 60G50 60J20)},
  MRNUMBER = {2838338},
       DOI = {10.1002/cpa.20382},
}

@article {Sznitman10,
    AUTHOR = {Sznitman, Alain-Sol},
     TITLE = {Vacant set of random interlacements and percolation},
   JOURNAL = {Ann. of Math. (2)},
  FJOURNAL = {Annals of Mathematics. Second Series},
    VOLUME = {171},
      YEAR = {2010},
    NUMBER = {3},
     PAGES = {2039--2087},
      ISSN = {0003-486X},
   MRCLASS = {60K35},
  MRNUMBER = {2680403},
MRREVIEWER = {Ingemar Kaj},
       DOI = {10.4007/annals.2010.171.2039},
}

@article {JO68,
    AUTHOR = {Jain, N. and Orey, S.},
     TITLE = {On the range of random walk},
   JOURNAL = {Israel J. Math.},
  FJOURNAL = {Israel Journal of Mathematics},
    VOLUME = {6},
      YEAR = {1968},
     PAGES = {373--380},
      ISSN = {0021-2172},
   MRCLASS = {60.66},
  MRNUMBER = {243623},
       DOI = {10.1007/BF02771217},
}

@article {JP70clt,
    AUTHOR = {Jain, Naresh C. and Pruitt, William E.},
     TITLE = {The central limit theorem for the range of transient random walk},
   JOURNAL = {Bull. Amer. Math. Soc.},
  FJOURNAL = {Bulletin of the American Mathematical Society},
    VOLUME = {76},
      YEAR = {1970},
     PAGES = {758--759},
      ISSN = {0002-9904},
   MRCLASS = {60.66},
  MRNUMBER = {258135},
MRREVIEWER = {V. K. Rohatgi},
       DOI = {10.1090/S0002-9904-1970-12536-8},
}

@article {JP70recur,
    AUTHOR = {Jain, Naresh C. and Pruitt, William E.},
     TITLE = {The range of recurrent random walk in the plane},
   JOURNAL = {Z. Wahrscheinlichkeitstheorie und Verw. Gebiete},
  FJOURNAL = {Zeitschrift f\"{u}r Wahrscheinlichkeitstheorie und Verwandte Gebiete},
    VOLUME = {16},
      YEAR = {1970},
     PAGES = {279--292},
   MRCLASS = {60.66},
  MRNUMBER = {281266},
MRREVIEWER = {H. Kesten},
       DOI = {10.1007/BF00535133},
}

@article {JP71,
    AUTHOR = {Jain, Naresh C. and Pruitt, William E.},
     TITLE = {The range of transient random walk},
   JOURNAL = {J. Analyse Math.},
  FJOURNAL = {Journal d'Analyse Math\'{e}matique},
    VOLUME = {24},
      YEAR = {1971},
     PAGES = {369--393},
      ISSN = {0021-7670},
   MRCLASS = {60.66},
  MRNUMBER = {283890},
MRREVIEWER = {J. R. Kinney},
       DOI = {10.1007/BF02790380},
}

@article {LeGall,
    AUTHOR = {Le Gall, J.-F.},
     TITLE = {Propri\'{e}t\'{e}s d'intersection des marches al\'{e}atoires. {I}.  {C}onvergence vers le temps local d'intersection},
   JOURNAL = {Comm. Math. Phys.},
  FJOURNAL = {Communications in Mathematical Physics},
    VOLUME = {104},
      YEAR = {1986},
    NUMBER = {3},
     PAGES = {471--507},
      ISSN = {0010-3616},
   MRCLASS = {60J15 (82A05)},
  MRNUMBER = {840748},
MRREVIEWER = {F. L. Spitzer},
}

@inproceedings {DE51,
    AUTHOR = {Dvoretzky, A. and Erd\"{o}s, P.},
     TITLE = {Some problems on random walk in space},
 BOOKTITLE = {Proceedings of the {S}econd {B}erkeley {S}ymposium on {M}athematical {S}tatistics and {P}robability, 1950},
     PAGES = {353--367},
 PUBLISHER = {University of California Press, Berkeley and Los Angeles},
      YEAR = {1951},
   MRCLASS = {60.0X},
  MRNUMBER = {0047272},
MRREVIEWER = {S. Kakutani},
}
